\documentclass[a4paper]{amsart}

\usepackage[maxbibnames=9]{biblatex}
\usepackage{color}
\usepackage{graphicx}
\usepackage[normalem]{ulem}
\usepackage{morefloats}
\usepackage{hyperref}
\numberwithin{equation}{section}
\hypersetup{
    bookmarks=true,         
    unicode=false,          
    pdftoolbar=true,        
    pdfmenubar=true,        
    pdffitwindow=false,     
    pdfstartview={FitH},    
    pdftitle={My title},    
    pdfauthor={Author},     
    pdfsubject={Subject},   
    pdfcreator={Creator},   
    pdfproducer={Producer}, 
    pdfkeywords={keyword1, key2, key3}, 
    pdfnewwindow=true,      
    colorlinks=false,       
    linkcolor=green,          
    citecolor=green,        
    filecolor=green,      
    urlcolor=green,           
    urlbordercolor={1 1 1}  
}

\usepackage{amsmath, amsfonts, amsthm, amssymb}
\usepackage{fancyhdr}

\usepackage{biblatex}

\usepackage{mathrsfs}
\usepackage{blkarray}
\usepackage{amsthm}
\usepackage{mathtools}
\usepackage[title]{appendix}
\newtheorem{theorem}{Theorem}[section]
\newtheorem*{theorem*}{Theorem}
\newtheorem{lemma}[theorem]{Lemma}

\newcommand{\DD}{\mathbb{D}}

\newcommand{\C}{\mathbb{C}}

\newcommand{\RR}{\Rightarrow}

\newtheorem{proposition}[theorem]{Proposition}

\newtheorem{corollary}[theorem]{Corollary}
\theoremstyle{definition}
\newtheorem{definition}[theorem]{Definition}

\theoremstyle{remark}

\theoremstyle{definition}
\newtheorem*{exmp}{Example} 
\theoremstyle{definition}
\newtheorem*{question}{Question}

\begin{document}

\title[\resizebox{3.9in}{!}{  \large SUBINNER-FREE OUTER FACTORIZATIONS ON AN ANNULUS}]{\large SUBINNER-FREE OUTER FACTORIZATIONS ON AN ANNULUS}

\author[\resizebox{1.35in}{!}{\small GEORGIOS TSIKALAS}]{\small GEORGIOS TSIKALAS}
\thanks{Partially supported by National Science Foundation Grant DMS 2054199}

\address{DEPARTMENT OF MATHEMATICS AND STATISTICS, WASHINGTON UNIVERSITY IN ST. LOUIS, ST. LOUIS, MO, 63136}
\email{gtsikalas@wustl.edu} 
\keywords{subinner, free outer, complete Pick space, annulus. \\  \hspace*{0.34 cm} I thank Professor John E. M\raise.5ex\hbox{c}Carthy, my thesis advisor, for his invaluable help during the preparation of this work. I also thank the Onassis Foundation for providing financial support for my PhD studies.}
\small
\begin{abstract}
    \small
    Recent work of Aleman, Hartz,  M\raise.5ex\hbox{c}Carthy and Richter generalizes the classical inner-outer factorization of Hardy space functions to the complete Pick space setting, establishing  an essentially unique ``subinner-free outer" factorization. In this note, we investigate certain special examples of such factorizations in the setting of the function space induced on the annulus $A_r=\{r<|z|<1\}$ by the complete Pick kernel $$k_{r}(\lambda,\mu):=\frac{1-r^2}{(1-\lambda\bar{\mu})(1-r^2/\lambda\bar{\mu})}.$$

\end{abstract}
\maketitle

 \section{INTRODUCTION} 
 \large 
 
 Let $\mathcal{H}$ be a Hilbert function space induced by a \textit{complete Pick} kernel $k$. Well-known examples of such spaces on the unit disk $\DD$ are the Hardy space $H^2(\DD)$, where $k_w(z)=\frac{1}{1-z\bar{w}}$, and the Dirichlet space $D$, where $k_w(z)=\frac{1}{z\bar{w}}\log\frac{1}{1-z\bar{w}}$. Another example is the Drury-Arveson space $H^2_d$ of analytic functions on the $d$-dimensional unit ball $\mathbb{B}_d$, where $k_w(z)=\frac{1}{1-\langle z,w\rangle_{\C^d}}$ and $d\in\mathbb{N}\cup\{\infty\}$. In \cite{freeouterdraft}, Aleman, Hartz, M\raise.5ex\hbox{c}Carthy and Richter prove a unique factorization result for functions in $\mathcal{H}$, involving the two classes of \textit{subinner} and \textit{free outer} functions. In particular, they show that for every $f\in\mathcal{H}\setminus\{0\}$, there exists a unique (up to unimodular constants) \textit{subinner-free outer pair} $(\phi, h)$ such that 
 $$f=\phi h.$$
Here, the subinner factor $\phi$ is a contractive multiplier of $\mathcal{H}$ such that $||\phi h||=||h||,$ while the free outer factor $h$ belongs to a special subclass of the cyclic vectors of $\mathcal{H}$ (to be defined below). Note that this factorization is the same (without the uniqueness assertion) as the one considered by Jury and Martin in \cite{Weakprod}, \cite{JurySmirnov}. In these papers, the authors worked in the setting of the free Fock space $\mathcal{F}^2_d$ in $d$ variables, where a free inner-outer factorization is known to hold (see \cite{AriasPopescu}, \cite{DavidsonPitts}). Since there exists a natural isometric embedding of $\mathcal{H}$ into $\mathcal{F}^2_d$ for some $d\in\mathbb{N}\cup\infty$ (see Section \ref{123}), this free inner-outer factorization then also applies to functions in the embedded space $\mathcal{H}$. \par
Now, in the setting of the Hardy space $H^2(\mathbb{D}$), subinner functions coincide with the classical inner functions. In the general case, subinner functions form a ``large" class of contractive multipliers of $\mathcal{H}$; Theorem 1.9 in \cite{freeouterdraft} asserts that every $\phi$ in the unit ball of the multiplier algebra of  $\mathcal{H}$ is a pointwise limit of subinner functions. We note that it may even happen that every function in the space is, up to renormalization, a subinner multiplier! For instance, consider the Sobolev space $\mathcal{W}_1^2$ of those functions on $[0,1]$ that are absolutely continuous and whose derivatives are square-summable. Agler \cite{Sobolev} showed that $\mathcal{W}_1^2$ is a complete Pick space. Theorem 1.5 of the same paper implies that every $f\in\mathcal{W}_1^2$ is a norm-attaining multiplier of $\mathcal{W}_1^2,$ hence a constant multiple of a subinner function. 
 \par On the other hand, free outer functions are much harder to come by. While they coincide with the usual outer functions in the setting of $H^2(\mathbb{D})$, in general they form a strict subset of the cyclic vectors of $\mathcal{H}$. Also, the product of two free outer functions isn't necessarily free outer (see e.g. \cite[Example 11.1]{freeouterdraft}). Another perhaps surprising fact is the existence of nonconstant functions that are simultaneously subinner and free outer. Simple examples of such functions can be found even in the Dirichlet space $D$; see the remark after the proof of Theorem 14.9 in \cite{freeouterdraft}. 
 \par It is fair to say that free outer functions are not well-understood yet. The situation is exacerbated by a striking lack of examples; indeed, besides the familiar Hardy space setting, it seems that only the free outer factors of low-degree polynomials have been worked out explicitly so far. In this note, we enlarge our pool of available examples by investigating subinner-free outer factorizations in the setting of a certain complete Pick space on an annulus, which we now define.
 \par Let $\mathscr{H}^2(A_r)$ denote the Hilbert function space induced on the annulus $A_r=\{r<|z|<1\}$ by the kernel 
$$k_{r}(\lambda,\mu):=\frac{1-r^2}{(1-\lambda\bar{\mu})(1-r^2/\lambda\bar{\mu})}.$$
$\mathscr{H}^2(A_r)$ is known to be a complete Pick space (see \cite[p. 1137]{Rochberg}) and can also be naturally associated with the class of operators $\mathcal{F}_r:=\{T\in\mathcal{B}(H): r^2T^{-1}(T^{-1})^*+TT^*\le r^2+1,\hspace{0.08 cm}\sigma(T)\subset A_r\}$, as noted by the author in \cite{tsikalas2021von}. Our main result, proved in Section \ref{50}, is the following:
\begin{theorem} \label{100}
For every $f\in H^2(\DD)\subset \mathscr{H}^2(A_r),$ the classical inner-outer factorization in $H^2(\DD)$ coincides (up to multiplication by unimodular constants) with the subinner-free outer factorization of $f$ in $\mathscr{H}^2(A_r)$.  
An analogous result holds for functions in $H^2(\DD_0)\footnote{Here, $H^2(\DD_0)$ is the Hardy space over the unbounded disk $\DD_0=\{|z|>r\}$ centered at infinity.}\subset \mathscr{H}^2(A_r).$
\end{theorem}
Theorem \ref{100} tells us that a function in $\mathscr{H}^2(A_r)$ that is also analytic on the unit disk is $\mathscr{H}^2(A_r)$-free outer if and only if it is $H^2(\mathbb{D})$-outer. We show that a similar result holds for $\mathscr{H}^2(A_r)$-subinner functions that are also analytic on $\mathbb{D}.$
\begin{theorem} \label{69}
Suppose $\phi\in \text{Mult}(\mathscr{H}^2(A_r))\cap \text{Hol}(\mathbb{D})=H^{\infty}(\mathbb{D}).$ Then, $\phi$ is $H^2(\DD)$-inner if and only if $\phi$ is $\mathscr{H}^2(A_r)$-subinner. 
An analogous result holds for $H^2(\DD_0)$-inner functions.
\end{theorem}
 Finally, Corollary \ref{16} (see Section \ref{50}) allows us to obtain new examples of free outer functions in $H_2^2$ by using the embedding of $\mathscr{H}^2(A_r)$ into $H_2^2$.

 \small
 \section{PRELIMINARIES}\label{123}
 \large

  Let $\mathcal{H}$ be a separable Hilbert function space on a non-empty set $X$ with reproducing kernel $k.$ Write $\text{Mult}(\mathcal{H}):=\{\phi:X\to\C: \phi f\in\mathcal{H} \text{ for all } f\in\mathcal{H}\}$ for the multiplier algebra of $\mathcal{H}$. Every multiplier $\phi$ defines a bounded linear operator $M_{\phi}\in\mathcal{B}(\mathcal{H})$ by $M_{\phi}(f)=\phi f.$ 
 Putting $||\phi||_{\text{Mult}(\mathcal{H})}=||M_{\phi}||_{\mathcal{B}(\mathcal{H})}$ turns $\text{Mult}(\mathcal{H})$ into a Banach algebra. If $f\in\mathcal{H}$, write $[f]$ for the multiplier invariant subspace generated by $f,$ i.e. the closure of $\text{Mult}(\mathcal{H})f$ in $\mathcal{H}$. $f$ is called \textit{cyclic}, if $[f]=\mathcal{H}.$ Also, define $P_f: \text{Mult}(\mathcal{H})\to\C$ to be the linear functional $P_f(\phi):=\langle \phi f, f\rangle$ and set 
 $$\mathcal{E}_f=\{g\in\mathcal{H}:P_f=P_g\}. $$\par

We say that $\mathcal{H}$ is a \textit{complete Pick space} if for every $r\in\mathbb{N}$ and every finite collection of points $z_1, z_2, \dots, z_n\in X$ and matrices $W_1, W_2,\dots, W_n\in M_r(\C)$, positivity of the $nr\times nr$-block matrix
 $$ \big[ k(z_i, z_j)(I_{\C^r}-W_iW_j^*)\big]^n_{i, j=1}.$$
 implies that there exists $\Phi\in M_r(\text{Mult}(\mathcal{H}))$ of norm at most $1$ such that
 $$\Phi(z_i)=W_i\hspace*{0.2 cm} (i=1, \dots, n). $$
 In this setting, $k$ is called a \textit{complete Pick kernel}. By a theorem of Agler and M\raise.5ex\hbox{c}Carthy \cite{CNPkernels}, a function 
 $k:X\times X\to\C$ is an (irreducible) complete Pick kernel if and only if it has 
 the form 
 \begin{equation} \label{1}  k_x(y)=k(y,x)=\frac{\delta(y)\overline{\delta(x)}}{1-\langle u(y),u(x)\rangle_{\C^d}},\end{equation}
 where $\delta: X \to \mathbb{C}-\{0\}$, $u$ is a function from $X$ into the open unit ball $\mathbb{B}_d$ of $\C^d$ and $d\in\mathbb{N}\cup\{\infty\}$  ($\mathbb{B}_{\infty}$ being the open unit ball of $l^2).$ We have already mentioned examples of well-known function spaces with CP kernels. In particular, the Drury-Arveson space $H^2_d$ with kernel $k_w(z)=\frac{1}{1-\langle z,w\rangle_{\C^d}}$ has a universal property among all CP spaces (see below for details). For further background on this topic, see \cite{Pick}. \par
We can now state the Aleman, Hartz, M\raise.5ex\hbox{c}Carthy, Richter factorization result. To be precise, their theorems are stated in the setting of a CP kernel $k$ \textit{normalized} at $z_0\in X$ (meaning that  in (\ref{1}) we have $\delta\equiv 1$ and $u(z_0)=0$). We summarize them below.
 \begin{definition} \label{4} (a) A function $f\in\mathcal{H}$ is called \textit{free outer}, if 
     $$|f(z_0)|=\sup\{|g(z_0)|: g\in\mathcal{E}_f\}.$$
     (b) A multiplier $\phi\in\text{Mult}(\mathcal{H})$ is called \textit{subinner}, if $||\phi||_{\text{Mult}(\mathcal{H})}=1$ and if there exists a nonzero $h\in\mathcal{H}$ such that with $||\phi h||=||h||.$ \\
     (c)  A pair ($\phi$, $f$) is called a \textit{subinner/free outer pair}, if $\phi$ is subinner, $f$ is free outer with $f(z_0)>0$, and $||\phi f||=||f||. $
     \end{definition}
 \begin{theorem} \label{3}
 For every $f\in\mathcal{H}\setminus\{0\}$ there is a unique subinner/free outer pair ($\phi$, $h$) such that $f=\phi h.$
 \end{theorem}
 \begin{theorem}
 If $f, g \in \mathcal{H}\setminus\{0\}$, then $P_f=P_g$ if and only if $f$ and $g$ have the same free outer factors. 
 \end{theorem} \noindent
  Having a distinguished normalization point turns out not to be important.
 \begin{corollary}\label{2}
 Let $h\in\mathcal{H}$. The following are equivalent: \\
 (a) $h$ is free outer, \\
 (b) there is $z\in X$ with $|h(z)|=\sup\{|f(z)|: f\in\mathcal{E}_f\}$, \\
 (c) for all $z\in X$ we have $|h(z)|=\sup\{|f(z)|: f\in\mathcal{E}_f\}$.
 \end{corollary} \par

We now briefly discuss preliminaries in regards to the free Fock space. Let $d\in\mathbb{N}\cup\{\infty\}$ and write $\mathbb{F}^+_d$ for the free semigroup on $d$ letters $\{1, 2,\dots\}$; that is, the set of all words $w=w_1w_2\cdots w_k$ over all (finite) lengths $k$, where each $w_j\in\{1, 2,\dots\}.$   We also include the empty word $\varnothing$ in $\mathbb{F}^+_d$, the length of which is defined to be zero.  If $w\in \mathbb{F}^+_d,$ then $\alpha(w)\in\mathbb{N}^d_0$ is the multi-index associated with $w$, defined by $\alpha(w)=(\alpha_1, \dots, \alpha_d)$, where $\alpha_j$ equals the number of times the letter $j$ occurs in $w.$ 
Also, let $x=(x_1,\dots,x_d)$ be a freely non-commuting indeterminate with $d$ components. If $w\in\mathbb{F}^+_d,$ then the free monomials are defined by $x^w=1,$ if $w=\varnothing,$ and $x^w=x_{w_1}\dots x_{w_k}$, if $w=w_1\cdots w_k.$ \par
The free Fock space $\mathcal{F}^2_d$ is the space of all power series in d non-commuting formal variables  with square-summable coefficients, i.e. $F\in \mathcal{F}^2_d$ if and only if $F(x)=\sum_{w\in\mathbb{F}^+_d}\hat{F}(w)x^w$ and $||F||^2=\sum_{w\in\mathbb{F}^+_d}|\hat{F}(w)|^2<\infty.$ A distinguished subspace of $\mathcal{F}^2_d$ is the symmetric Fock space $\mathcal{H}^2_d\subseteq \mathcal{F}^2_d$. An element $F\in\mathcal{F}^2_d$ is in $\mathcal{H}^2_d$ if and only if $\hat{F}(w)=\hat{F}(v),$ whenever $\alpha(w)=\alpha(v).$ Now, for every $z=(z_1, z_2, \dots, z_d)\in\mathbb{B}_d$ and $\mathbf{n}:=(n_1, n_2, \dots, n_d)\in\mathbb{N}^d_0,$  set $z^{\mathbf{n}}:=z_1^{n_1}\cdots z_d^{n_d}$. The map $T:H^2_d\to\mathcal{F}^2_d$ defined by \begin{equation}\label{30}h(z)=\sum_{\mathbf{n}\in\mathbb{N}^d_0}\frac{h_{\mathbf{n}}}{||z^{\mathbf{n}}||^2_{H^2_d}}z^{\mathbf{n}}\mapsto H(x):=\sum_{\mathbf{n}\in\mathbb{N}^d_0} h_{\mathbf{n}}\Bigg(\sum_{w|\text{ }\lambda(w)=\mathbf{n}}x^w\Bigg) \end{equation}
is an isometric embedding of $H^2_d$ into $\mathcal{F}^2_d$ ($T$ identifies $H^2_d$ with $\mathcal{H}^2_d$, see \cite{shalithandbook}, Section 4).
Also, if $k_z(\lambda)=\frac{1}{1-\langle\lambda, z\rangle}$ denotes the kernel of $H^2_d,$ we set $K_z:=Tk_z=\sum_{\mathbf{n}\in\mathbb{N}^d_0}\overline{z}^{\mathbf{n}}\sum_{w|\text{ }\lambda(w)=\mathbf{n}}x^w\in\mathcal{H}^2_d,$ for every $z\in\mathbb{B}_d.$ For further details about the free Fock space and related non-commutative function theory, see \cite{JurySmirnov}, \cite{Pop1}, \cite{Pop2}, \cite{Pop3}, \cite{Shalit} and \cite{Salomon}.\par
 Now, consider an arbitrary normalized CP kernel $k_x(y)=\frac{1}{1-\langle u(y),u(x)\rangle}$, where $u:X\to\mathbb{B}_d$ and $d\in \mathbb{N}\cup\{\infty\}$. We can identify (see \cite{CNPkernels}) the Hilbert function space $\mathcal{H}_k$ associated with $k$ with the subspace 
$$\mathcal{H}=\text{ closed linear span of } \{K_z:z\in \text{ran }  u\}\subset \mathcal{H}^2_d\subset \mathcal{F}^2_d.$$
Furthermore, it can be shown that $\mathcal{H}$ is invariant under adjoints of multipliers of $\mathcal{F}^2_d$ and also the map $Uk_x=K_{u(x)}$ extends to be a linear isometry $U:\mathcal{H}_k\to\mathcal{F}^2_d$ with range equal to $\mathcal{H}.$ Thus, $UU^*=P_{\mathcal{H}}$ and we have $U^*=C_u,$ where $C_u F(x)=F(u(x))$ for all $x\in X.$ \par
Denote by $\mathcal{LF}^{\infty}_d$ the algebra of all elements $G\in\mathcal{F}^2_d$ such that the operator $F\mapsto GF$ (multiplication by $G$ from the left) is bounded on $\mathcal{F}^2_d.$ The algebra $\mathcal{RF}^{\infty}_d$ of bounded right multiplication operators is defined analogously. An element $F\in\mathcal{F}^{2}_d$ will be called \textit{left-outer} if $\{GF: G\in\mathcal{LF}^{\infty}_d \}$ is dense in $\mathcal{F}^2_d.$ Similarly, $F$ will be called \textit{right-outer} if $\{FG: G\in\mathcal{RF}^{\infty}_d \}$ is dense in $\mathcal{F}^2_d.$ The following is the content of Theorem 5.3 and Lemma 6.1 in \cite{freeouterdraft}.
\begin{theorem} \label{15}
Suppose $\mathcal{H}_k$ is a CP space and let $U: \mathcal{H}_k\to\mathcal{F}^2_d$ be its natural embedding into $\mathcal{F}^2_d$. If $f\in\mathcal{H}_k,$ the following are equivalent:
\begin{itemize}
    \item[a)] $f$ is free outer in $\mathcal{H}_k$;
    \item[b)] $Uf$ is left-outer in $\mathcal{F}^2_d$;
     \item[c)] $Uf$ is right-outer in $\mathcal{F}^2_d$;
     \item[d)] $S^{-1}(Uf)$ is free outer in $H^2_d,$
     where $S:=P_{\mathcal{H}^2_d}T$ is the isometric identification (\ref{30}) of $H^2_d$ with $\mathcal{H}^2_d.$
\end{itemize}
\end{theorem}
Finally, we mention a few basic facts about the space $\mathscr{H}^2(A_r)$ that will be needed. Let $A_r=\{r<|z|<1\}$ denote an annulus. Recall that $\mathscr{H}^2(A_r)$ is defined as the Hilbert function space on $A_r$ induced by the kernel $$k_{r}(\lambda,\mu):=\frac{1-r^2}{(1-\lambda\bar{\mu})(1-r^2/\lambda\bar{\mu})}.$$
Letting $||^.||_{\mathscr{H}^2(A_r)}$ denote the corresponding norm, we obtain that $$||f||_{\mathscr{H}^2(A_r)}^2=\sum_{-\infty}^{-1}r^{2n}|c_n|^2+\sum_{0}^{\infty}|c_n|^2,$$
for every $f=\sum c_nz^n\in\mathscr{H}^2(A_r).$
Also, every $\phi\in \textnormal{Mult}(\mathscr{H}^2(A_r))$ satisfies (see \cite{tsikalas2021von})
 $$||\phi||_{\infty}\le ||\phi||_{\textnormal{Mult}(\mathscr{H}^2(A_r))}\le \sqrt{2}||\phi||_{\infty}, $$
 hence $\textnormal{Mult}(\mathscr{H}^2(A_r))$ is equal (but not isometric) to $H^{\infty}(A_r).$
\\
Now, putting 
$$u:A_r\to\mathbb{B}_2$$
$$z\mapsto \bigg(\frac{z}{\sqrt{r^2+1}},\frac{r}{\sqrt{r^2+1}}\frac{1}{z}\bigg),$$
and letting $k_2$ denote the kernel of $H^2_2,$ it can be easily checked that 
\begin{equation} \label{20} k_{r}(\lambda, \mu)=\bigg(\frac{1-r^2}{1+r^2}\bigg) k_2(u(\lambda), u(\mu)), \hspace{0.2 cm} \forall \lambda, \mu\in A_r.\end{equation}
So, $\mathscr{H}^2(A_r)$ is a CP space and it should be mentioned that, while $k_r$ is not a normalized CP kernel, normalizing it only requires dividing by the constant $(1-r^2)/(1+r^2)$. Also, Definition \ref{4} tells us that replacing the norm of the base space $\mathcal{H}_k$ by any constant multiple of it does not affect the properties of being subinner or free outer. Hence, the fact that $k_r$ is not normalized is not important for our results. 
\small
\section{MAIN RESULTS}\label{50}
\large
First, we record a helpful lemma; for multipliers of $\mathscr{H}^2(A_r)$ that are analytic either on $\DD$ or on $\DD_0=\{r<|z|\}$, the multiplier norm actually coincides with the supremum norm. 
\begin{lemma} \label{aaa}
\begin{itemize}
    \item[i)] If $f\in H^2(\DD),$
 then $$||f||_{H^2(\DD)}=||f||_{\mathscr{H}^2(A_r)}.$$
 Similarly, if $f\in H^2(\DD_0),$
 then $$||f||_{H^2(\DD_0)}=||f||_{\mathscr{H}^2(A_r)}.$$
 
 \item[ii)] If either $\phi\in H^{\infty}(\DD)$ or $\phi\in H^{\infty}(\DD_0),$ then
 $$||\phi||_{\text{Mult}(\mathscr{H}^2(A_r))}=||\phi||_{H^{\infty}(A_r)}. $$

 \end{itemize}
\end{lemma}
\begin{proof}
Part i) is obvious so we only show part ii). \\ WLOG, suppose $||\phi||_{H^\infty(A_r)}\le 1.$ If  $\phi\in \text{Hol}(\DD)$, then $||\phi||_{\infty}\le 1$ as a function on $\DD$ (by the maximum modulus principle) and so $\phi$ is a contractive multiplier of $H^2(\DD).$ 
By \cite[Theorem 5.21]{PauRa}, this implies that 
$$(1-\phi(\lambda)\overline{\phi(\mu)})\frac{1}{1-\lambda\bar{\mu}}\ge 0 \text{ \hspace{0.1 cm}on \hspace{0.1 cm}} \DD\times \DD $$
$$\Longrightarrow (1-\phi(\lambda)\overline{\phi(\mu)})\frac{1}{1-\lambda\bar{\mu}}\ge 0 \text{ \hspace{0.1 cm}on \hspace{0.1 cm}} A_r\times A_r $$
$$\Longrightarrow (1-\phi(\lambda)\overline{\phi(\mu)})\frac{1}{\big(1-\frac{r^2}{\lambda\bar{\mu}}\big)\big(1-\lambda\bar{\mu}\big)}\ge 0 \text{ \hspace{0.1 cm}on \hspace{0.1 cm}} A_r\times A_r,$$ 
as the Schur product of two positive semi-definite kernels is positive semi-definite (notice that $1/(1-r^2/(\lambda\bar{\mu}))$ is the kernel of the Hardy space on $\DD_0$). Hence, we conclude that 
$$(1-\phi(\lambda)\overline{\phi(\mu)})k_{r}(\lambda,\mu)\ge 0 $$ 
$$\Longrightarrow ||\phi||_{\text{Mult}(\mathscr{H}^2(A_r))}\le 1.$$
We have shown that $||\phi||_{\text{Mult}(\mathscr{H}^2(A_r))}\le||\phi||_{H^{\infty}(A)}, $ which concludes our argument. For the case where $\phi\in \text{Hol}(\DD_0),$
the proof proceeds in an analogous manner.
\end{proof}
We proceed to show that the classical inner functions of the Hardy space on $\DD$ become subinner functions when viewed as multipliers of $\mathscr{H}^2(A_r)$.
\begin{proposition} \label{12}
Suppose $g\in H^{\infty}(\DD)$ is an $H^2(\DD)$-inner function. Then, $g$ is a subinner multiplier of $\mathscr{H}^2(A_r)$. An analogous result holds if $g\in H^{\infty}(\DD_0)$ and $g$ is $H^2(\DD_0)$-inner\footnote{i.e. there exists an $H^2(\DD)$-inner function $h$ such that $g(z)=h(r/z).$}.
\end{proposition}
\begin{proof}
We only consider the case where $g$ is $H^2(\DD)$-inner. By Lemma \ref{aaa}, we have
$$||g||_{\text{Mult}(\mathscr{H}^2(A_r))}=||g||_{H^{\infty}(A_r)}=||g||_{H^{\infty}(\DD)}=1.$$
Now, let $h\in H^2(\DD)$ be arbitrary. Then, 
$$||gh||_{\mathscr{H}^2(A_r)}=||gh||_{H^2(\DD)}=||h||_{H^2(\DD)}=||h||_{\mathscr{H}^2(A_r)}. $$
Thus, $g$ is a norm-attaining multiplier of $\mathscr{H}^2(A_r)$ with multiplier norm equal to $1,$ i.e. a subinner multiplier.
\end{proof}

Our next step will be to show that for functions in $H^2(\DD)$ (similarly, for functions in $H^2(\DD_0)$), the property of $\mathscr{H}^2(A_r)$-free outerness coincides with outerness in the classical Hardy space sense. To do this, we will be needing the following two lemmata. \\ Let $\text{Mult}(H^2(\DD),\mathscr{H}^2(A_r))$ denote the set of functions $\phi: A_r\to\C$ that multiply $H^2(\DD)$ (or, more precisely, the restrictions of $H^2(\DD)$ functions to $A_r$) boundedly into $\mathscr{H}^2(A_r).$ These multipliers turn out to be equal precisely to all functions $\phi=\sum c_n z^n\in\text{Hol}(A_r)$ such that $\sum_{n>0} c_n z^n\in H^{\infty}(\DD)$ and $\sum_{n<0} c_n z^n\in H^2(\DD_0)$.

\begin{lemma} \label{8}
$$\text{Mult}(H^2(\DD),\mathscr{H}^2(A_r))=\{f+g: f\in H^2(\DD_0), g\in H^{\infty}(\DD)\}.$$
Also, 
$$||f+g||_{\text{Mult}(H^2(\DD),\mathscr{H}^2(A_r))}\le \frac{||f||_{H^2(\DD_0)}}{(1-r^2)}+||g||_{H^{\infty}(\DD)}.$$
An analogous result holds for $\text{Mult}(H^2(\DD_0),\mathscr{H}^2(A_r)).$
\end{lemma}
\begin{proof} (Note that there is a slight abuse of notation here, as we are identifying $H^2(\DD)$ with the Hilbert function space consisting of the restrictions of all $H^2(\DD)$ functions to $A_r.$ The kernel of that space is, of course, the restriction of the Szegő kernel for $\DD$ to $A_r\times A_r.)$\\
Put $M=\text{Mult}(H^2(\DD),\mathscr{H}^2(A_r)).$ 
Since $1\in H^2(\DD),$ we obtain $M\subset \mathscr{H}^2(A_r).$ Now, let $h\in H^2(\DD_0)$.  By \cite[Theorem 3.11]{PauRa}, we obtain 
$$h(\lambda)\overline{h(\mu)}\le ||h||_{H^2(\DD_0)}k_{H^2(\DD_0)}(\lambda,\mu) \text{ in } \DD_0\times\DD_0 $$
$$\RR h(\lambda)\overline{h(\mu)}\le ||h||_{H^2(\DD_0)}k_{H^2(\DD_0)}(\lambda,\mu) \text{ in } A_r\times A_r               $$
$$\RR h(\lambda)\overline{h(\mu)}k_{H^2(\DD)}(\lambda,\mu)\le ||h||_{H^2(\DD_0)}k_{H^2(\DD_0)}(\lambda,\mu)k_{H^2(\DD)}(\lambda,\mu) \text{ in } A_r\times A_r                $$
$$\RR h(\lambda)\overline{h(\mu)}k_{H^2(\DD)}(\lambda,\mu)\le \frac{||h||_{H^2(\DD_0)}}{(1-r^2)}k_r(\lambda,\mu) \text{ in } A_r\times A_r.                $$
By \cite[Theorem 5.21]{PauRa}, this last positivity condition implies that $h\in M$ and also that $||h||_{M}\le ||h||_{H^2(\DD_0)}/(1-r^2).$  This gives us $H^2(\DD_0)\subset M.$ \\
Now, consider $g\in H^{\infty}(\DD).$ For every $f\in H^2(\DD),$ we have 
$$||gf||_{\mathscr{H}^2(A_r)}=||gf||_{H^2(\DD)}\le ||g||_{H^{\infty}(\DD)}||f||_{H^2(\DD)}. $$
This shows that $g\in M$ and also that $||g||_{M}=||g||_{H^{\infty}(\DD)}.$ We conclude that $H^2(\DD_0)+H^{\infty}(\DD)\subset M.$ \\
For the converse, let $f\in M\subset \mathscr{H}^2(A_r).$ Choose any function $f_1\in H^2(\DD)$ with the property that $f-f_1\in H^2(\DD_0).$ By our previous observations, we obtain $f-f_1\in M.$ Also, since $f_1=f-(f-f_1)\in M\cap H^2(\DD)$ and $H^2(\DD)$ is contained isometrically in $\mathscr{H}^2(A_r),$ we have $f_1\in\text{Mult}(H^2(\DD))\subset M$ and also  $||f_1||_{M}=||f_1||_{H^{\infty}(\DD)}.$ We can thus write:
$$||f||_{M}=||f_1+(f-f_1)||_{M}\le ||f_1||_{M}+||f-f_1||_{M}\le ||f_1||_{H^{\infty}(\DD)}+\frac{||f-f_1||_{H^2(\DD_0)}}{1-r^2}.$$
\end{proof}

\begin{lemma} \label{9}
Let $f,g \in H^2(\DD)$ and suppose that $$\langle\phi f,f\rangle_{H^2(\DD)}=\langle\phi g,g\rangle_{H^2(\DD)}, \text{ for all } \phi\in H^{\infty}(\DD).$$
Then, 
 $$\langle\phi f,f\rangle_{\mathscr{H}^2(A_r)}=\langle\phi g,g\rangle_{\mathscr{H}^2(A_r)}, \text{ for all } \phi\in \text{Mult}(\mathscr{H}^2(A_r)).$$
\end{lemma}
\begin{proof}
Suppose $f,g$ satisfy the given assumptions and let $\phi \in H^{\infty}(A_r)=\text{Mult}(\mathscr{H}^2(A_r)).$ Write $\phi=\phi_1+\phi_2,$ where $\phi_1\in {H^2(\DD)}$ and $\phi_2=\sum_{-\infty}^{n=-1}a_n z^n\in \frac{1}{z}H^2(\DD_0).$ By the maximum modulus principle, we obtain  $\phi_1 \in H^{\infty}(\DD), \phi_2\in H^{\infty}(\DD_0).$ 
Now, notice that 
$$\lim_{k\to -\infty}\sum_{k}^{n=-1}a_n z^n f=\Bigg(\sum_{-\infty}^{n=-1}a_n z^n \Bigg)f$$
in the $||^.||_{\mathscr{H}^2(A_r)}$ norm, as $H^2(\DD)\subset \text{Mult} (H^2(\DD_0), \mathscr{H}^2(A_r)),$ due to Lemma \ref{8}. We can write 
$$\langle \phi f,f\rangle_{\mathscr{H}^2(A_r)}=\langle \phi_1 f,f\rangle_{\mathscr{H}^2(A_r)}+\langle \phi_2 f,f\rangle_{\mathscr{H}^2(A_r)},$$ where

    $$\langle \phi_1 f,f\rangle_{\mathscr{H}^2(A_r)}=\langle \phi_1 f,f\rangle_{H^2(\DD)} $$
    $$=\langle \phi_1 g,g\rangle_{H^2(\DD)} \text{ (by assumption) } $$
    $$=\langle \phi_1 g,g\rangle_{\mathscr{H}^2(A_r)}$$
    and also 
    $$\langle \phi_2 f,f\rangle_{\mathscr{H}^2(A_r)} =\Bigg\langle \Bigg(\lim_{k\to -\infty}\sum_{k}^{n=-1}a_n z^n \Bigg)f,f\Bigg\rangle_{\mathscr{H}^2(A_r)}$$
    $$=\lim_{k\to -\infty} \sum_{k}^{n=-1}a_n \langle z^n f,f\rangle_{\mathscr{H}^2(A_r)} $$
    
    $$=\lim_{k\to -\infty} \sum_{k}^{n=-1}a_n \langle  f,z^{-n}f\rangle_{H^2(\DD)}  $$
    $$=\lim_{k\to -\infty} \sum_{k}^{n=-1}a_n \langle  g,z^{-n}g\rangle_{H^2(\DD)}  ,$$
    where the last equality is because of our initial assumptions. Working our way in the opposite direction, we can then show that 
    $$\lim_{k\to -\infty} \sum_{k}^{n=-1}a_n \langle  g,z^{-n}g\rangle_{H^2(\DD)}=\langle \phi_2 g,g\rangle_{\mathscr{H}^2(A_r)} .$$
Hence, $$\langle \phi f,f\rangle_{\mathscr{H}^2(A_r)}=\langle \phi_1 g,g\rangle_{\mathscr{H}^2(A_r)}+\langle \phi_2 g,g\rangle_{\mathscr{H}^2(A_r)}=\langle \phi g,g\rangle_{\mathscr{H}^2(A_r)},$$which concludes the proof.
\end{proof}

We can now characterize those functions $f\in H^2(\DD)\subset\mathscr{H}^2(A_r)$ that are $\mathscr{H}^2(A_r)$-free outer.
\begin{theorem} \label{11}
Suppose $f\in H^2(\DD).$ Then, $f$ is $H^2(\DD)$-outer if and only if $f$ is $\mathscr{H}^2(A_r)$-free outer. 
An analogous result holds for $H^2(\DD_0)$-outer functions.
\end{theorem}
\begin{proof}
Fix an arbitrary $z_0\in A_r.$\\
First, suppose that $f\in H^2(\DD)$ is not $H^2(\DD)$-outer and also let $g$ denote the $H^2(\DD)$-outer factor of $f.$ Since $|f|=|g|$ on $\partial\mathbb{D},$ we  obtain
$$\langle \phi f,f\rangle_{H^2(\DD)}=\langle \phi g,g\rangle_{H^2(\DD)}, \text { for all } \phi\in H^{\infty}(\DD),$$
and also $|g(z_0)|>|f(z_0)|.$
But then, Lemma \ref{9} tells us that 
 $$\langle\phi f,f\rangle_{\mathscr{H}^2(A_r)}=\langle\phi g,g\rangle_{\mathscr{H}^2(A_r)}, \text{ for all } \phi\in H^{\infty}(A_r)$$
 $$\RR P_{f}=P_{g} \hspace{0.1 cm} \text{ in } \mathscr{H}^2(A_r).$$
 Since $|g(z_0)|>|f(z_0)|,$ Corollary \ref{2} implies that $f$ is not $\mathscr{H}^2(A_r)$-free outer. \\
 Conversely, suppose that $f$ is $H^2(\DD)$-outer. Suppose also that $f$ is not $\mathscr{H}^2(A_r)$-free outer. We will reach a contradiction.\\
 Write $f=gh$ for the subinner-free outer factorization of $f$ in $\mathscr{H}^2(A_r).$ We claim that $h\notin H^2(\DD).$ \\
 Indeed, if $h\in H^2(\DD)$ we can write 
 $$ P_{f}=P_{h} \hspace{0.1 cm} \text{ in } \mathscr{H}^2(A_r) \text{ (as $h$ is the free outer factor of $f$) } $$
 $$\RR \langle\phi f,f\rangle_{\mathscr{H}^2(A_r)}=\langle\phi h,h\rangle_{\mathscr{H}^2(A_r)}, \text{ for all } \phi\in H^{\infty}(A_r)  $$
 $$\RR \langle\phi f,f\rangle_{H^2(\DD)}=\langle\phi h,h\rangle_{H^2(\DD)}, \text{ for all } \phi\in H^{\infty}(\DD) $$
 $$\RR P_{f}=P_{g} \hspace{0.1 cm} \text{ in } H^2(\DD)$$
 $$\RR |f(z)|\ge |h(z)|, \text{ for all } z\in \DD, $$
 as $f$ is $H^2(\DD)$-outer. But $h$ is the free outer factor of $f$ and $f$ is \textit{not} $\mathscr{H}^2(A_r)$-free outer, so Corollary \ref{2} implies that $|h(z_0)|>|f(z_0)|, $ a contradiction. \\
 Hence, $h\notin H^2(\DD).$ This implies the existence of $k<0$ such that the term $a_k z^k$ in the Laurent expansion of $h$ is nonzero. Thus $||zh||_{\mathscr{H}^2(A_r)}<||h||_{\mathscr{H}^2(A_r)}.$ But then, we obtain 
$$||f||_{\mathscr{H}^2(A_r)}=||f||_{H^2(\DD)}=||zf||_{H^2(\DD)}$$
$$=||zf||_{\mathscr{H}^2(A_r)}=||g(zh)||_{\mathscr{H}^2(A_r)}\le ||zh||_{\mathscr{H}^2(A_r)} $$
$$<||h||_{\mathscr{H}^2(A_r)}=||f||_{\mathscr{H}^2(A_r)}, $$
a contradiction again. Thus, $f$ must be $\mathscr{H}^2(A_r)$-free outer and our proof is complete.

\end{proof}
We now show Theorem \ref{100} from the Introduction. 
\begin{proof}[Proof of Theorem \ref{100}]
Let $f\in H^2(\DD)$ and write $f=\phi h$ for the classical inner-outer factorization in $H^2(\DD).$ By Theorem \ref{11}, $h$ is $\mathscr{H}^2(A_r)$-free outer, while $\phi$ is subinner by Proposition \ref{12}. Notice also that $||f||_{\mathscr{H}^2(A_r)}=||h||_{\mathscr{H}^2(A_r)}$ and so $f=\phi h$ coincides (up to multiplication by unimodular constants) with the subinner-free outer factorization of $f$ in $\mathscr{H}^2(A_r)$.
\end{proof}

We have proved that a function in $\mathscr{H}^2(A_r)$ that is also analytic on the unit disk is $\mathscr{H}^2(A_r)$-free outer if and only if it is $H^2(\mathbb{D})$-outer. Theorem \ref{69} says that an analogous result holds for $\mathscr{H}^2(A_r)$-subinner functions that are also analytic on $\mathbb{D}.$ The key result used in the proof is a modified subinner-free outer factorization
for functions in spaces whose reproducing kernel has a CP
factor.

\begin{proof}[Proof of Theorem \ref{69}]
Suppose $\phi\in H^{\infty}(\mathbb{D})$. \par
If $\phi$ is $H^2(\DD)$-inner, then it must also be $\mathscr{H}^2(A_r)$-subinner, by Proposition \ref{12}. \par 
Now, suppose that $\phi$ is $\mathscr{H}^2(A_r)$-subinner. Thus, $||\phi||_{H^{\infty}(\mathbb{D})}=1$ (by Lemma \ref{aaa}) and also there exists a nonzero $f\in \mathscr{H}^2(A_r)$ such that $$||\phi f||_{\mathscr{H}^2(A_r)}=||f||_{\mathscr{H}^2(A_r)}.$$ Letting $s$ denote the classical Sz\'{e}go kernel on $\mathbb{D},$ it is easy to see that $k_r/s$ is positive semi-definite on $A_r\times A_r.$ Hence, by \cite[Theorem 1.10]{freeouterdraft}, we can find a (unique, up to multiplication by unimodular constants) pair of nonzero functions $\psi\in\text{Mult}(H^2(\mathbb{D}), \mathscr{H}^2(A_r))$ and $h\in H^2(\mathbb{D})$ such that 
\begin{itemize}
    \item [i)]
    $f=\psi h,$
    
     \item [ii)]
     $||h||_{H^2(\mathbb{D})}=||f||_{\mathscr{H}^2(A_r)},$
     
      \item [iii)]
      $h$ is $H^2(\mathbb{D})$-free outer,

       \item [iv)]
       $||\psi||_{\text{Mult}(H^2(\mathbb{D}), \mathscr{H}^2(A_r))}\le 1.$
\end{itemize}
Since $\phi h\in H^2(\mathbb{D})$, we can write 
$$||f||_{\mathscr{H}^2(A_r)}=||\phi f ||_{\mathscr{H}^2(A_r)}=|| \phi \psi h||_{\mathscr{H}^2(A_r)}=||\psi(\phi h) ||_{\mathscr{H}^2(A_r)} $$
$$\le ||\phi h ||_{H^2(\mathbb{D})} \hspace*{1.5 cm} \text{(by item (iv))}$$
$$\le ||h ||_{H^2(\mathbb{D})} \hspace*{0.5 cm} \text{(since $||\phi||_{H^{\infty}(\mathbb{D})}=1$)}$$
$$= ||f||_{\mathscr{H}^2(A_r)} \hspace*{1.5 cm} \text{(by item (ii))}$$
Thus, we must have $||\phi h||_{H^2(\mathbb{D})}=||h||_{H^2(\mathbb{D})}$, which implies that $\phi$ is $H^2(\mathbb{D})$-inner.

\end{proof}

\par

We now combine the embedding (\ref{20}) of $\mathscr{H}^2(A_r)$ into $H^2_2$ with Theorem \ref{100} to obtain new examples of free outer functions in $H^2_2$.

\begin{corollary} \label{16}
Let $f=\sum_{n=0}^{\infty}a_n z^n$ be an $H^2(\DD)$-outer function. Then, for every $r\in(0,1),$ the function $F\in H^2_2$ uniquely defined by  
$$\langle F,z^i_1 z^j_2\rangle_{H^2_2}=\frac{1-r^2}{1+r^2}\frac{r^j}{\big(\sqrt{r^2+1}\big)^{i+j}}a_{i-j}, \hspace{0.2 cm} (\text{for all } i,j\ge 0)$$
is $H_2^2$-free outer (put $a_{k}=0$ for negative $k$). 
Hence, if $T: H^2_d \to \mathcal{F}^2_d$ denotes the natural isometric embedding of $H^2_d$ into $\mathcal{F}^2_d$, $T(F)$ is both left- and right-outer in $\mathcal{F}^2_d.$
An analogous result holds for $H^2(\DD_0)$-outer functions.
\end{corollary}
 \begin{proof}
We only show the $H^2(\DD)$ version. Suppose $f=\sum_{n=0}^{\infty}a_n z^n$ is $H^2(\DD)$-outer. By our previous results, $f$ will be $\mathscr{H}^2(A_r)$-free outer. Consider the rescaled kernel defined on $A_r\times A_r$ by
\begin{equation} \label{40} k'_r(\lambda,\mu):=\frac{1+r^2}{1-r^2}k_r(\lambda,\mu)=\frac{1+r^2}{(1-\lambda\overline{\mu})(1-r^2/\lambda\overline{\mu})}=k_2(u(\lambda), u(\mu)),\end{equation}
where $k_2$ is the kernel of $H^2_2.$ $k'_r$ is now a normalized CNP kernel. It induces the holomorphic function space  $\mathcal{H}_{k'_r}$ on $A_r,$ the norm of which is simply the norm of $\mathscr{H}^2(A_r)$ multiplied by $\sqrt{(1-r^2)/(1+r^2)}$. We easily see that $f$ will be $\mathcal{H}_{k'}$-free outer as well. \\
Now, recall the isometric identification $S=P_{\mathcal{H}^2_d}T$ (see (\ref{30})) of $H^2_d$ with the symmetric Fock space $\mathcal{H}^2_d$. By Theorem \ref{15}, we obtain that $F:=S^{-1}Uf$ is $H^2_2$-free outer, where $U:\mathcal{H}_{k'}\to\mathcal{F}^2_2$ denotes the embedding of $\mathcal{H}_{k'}$ into $\mathcal{H}^2_2$ induced by $(\ref{40}).$ This function satisfies (for all $i,j\ge 0$):
$$\langle F,z^i_1z^j_2\rangle_{H^2_2}=\langle F,S^{-1}S(z^i_1z^j_2)\rangle_{H^2_2}$$ $$=\langle f,U^*S(z^i_1z^j_2)\rangle_{\mathcal{H}_{k'}}  $$
$$=\bigg\langle f, \bigg(\frac{z}{\sqrt{r^2+1}}\bigg)^i\bigg(\frac{r}{\sqrt{r^2+1}}\frac{1}{z}\bigg)^j  \bigg\rangle_{\mathcal{H}_{k'}} $$
$$=\frac{1-r^2}{1+r^2}\bigg\langle f,\frac{r^j}{\big(\sqrt{r^2+1}\big)^{i+j}}z^{i-j} \bigg\rangle_{\mathscr{H}^2(A_r)} $$
$$=\frac{1-r^2}{1+r^2}\frac{r^j}{\big(\sqrt{r^2+1}\big)^{i+j}}a_{i-j},$$
for any $r\in (0,1),$ as desired (for negative $k$, interpret $a_k$ as being zero). The rest of the Corollary follows immediately from Theorem \ref{15}.
\end{proof}

\begin{exmp} \label{90}
Consider the function $f(z)=z-\lambda.$ This function is 
\begin{itemize}
\item $\mathscr{H}^2(A_r)$-subinner if and only if $\lambda=0$ (by Theorem \ref{69});
    \item  $\mathscr{H}^2(A_r)$-free outer, if $|\lambda|\ge 1$  (by Theorem \ref{11});
    \item not $\mathscr{H}^2(A_r)$-cyclic (and hence \textit{not} $\mathscr{H}^2(A_r)$-free outer), if $r<|\lambda|<1;$
    \item $\mathscr{H}^2(A_r)$-cyclic but \textit{not} $\mathscr{H}^2(A_r)$-free outer, if $|\lambda|\le r$ (by Theorem \ref{11}).
\end{itemize} 
\end{exmp}
To conclude, we have a complete characterization of subinner and free outer functions in $\mathscr{H}^2(A_r)\cap\text{Hol}(\DD)$. Of course, it would be even more interesting if we were able to describe generic subinner and/or free outer functions in $\mathscr{H}^2(A_r).$ We pose this as a question. 

\begin{question}
Let $f\in \mathscr{H}^2(A_r)$. What is the subinner-free outer factorization of $f$?
\end{question}

\printbibliography

\end{document}